\documentclass[11pt,oneside,a4paper]{article}

\usepackage[english]{babel}
\usepackage[utf8]{inputenc} 
\usepackage{amsmath,amsfonts,amsthm} 
\usepackage{geometry}		
\usepackage{authblk}
\usepackage{graphicx}
\usepackage{subfig}
\usepackage{fourier} 
\usepackage{times}

\newcommand{\grad}{\mathrm{grad}}
\newtheorem{theo}{Theorem}

\newtheorem{prop}{Proposition}
\theoremstyle{definition}
\newtheorem{definition}{Definition} 
\theoremstyle{definition}
 
\theoremstyle{definition}

\newcommand{\x}{\ensuremath{x}}

\newcommand{\uv}{\ensuremath{u}}
\newcommand{\vv}{\ensuremath{v}}
\newcommand{\p}{\ensuremath{p}}
\newcommand{\kk}{\ensuremath{k}}
\newcommand{\h}{\ensuremath{{h}}}
\newcommand{\dd}{\ensuremath{T}}
\newcommand{\e}{\mathrm{e}}
\newcommand{\ce}{\ensuremath{c}}
\newcommand{\s}{\ensuremath{s}}
\newcommand{\cc}{\mathsf{c}}
\newcommand{\sh}{\ensuremath{\widetilde{s}}}
\numberwithin{equation}{section}
\DeclareMathAlphabet{\mathcal}{OMS}{cmsy}{m}{n}

\usepackage[activate={true,nocompatibility},
            final,tracking=true,
            kerning=true,
            spacing=true,
            factor=1100,
            stretch=10,
            shrink=10]{microtype}
\microtypecontext{spacing=nonfrench}

\title{Energy preserving methods on Riemannian manifolds}
\author{Elena Celledoni$^*$, Sølve Eidnes$^*$, Brynjulf Owren$^*$, Torbjørn Ringholm%
  \thanks{Department of Mathematical Sciences, Norwegian University of Science and Technology, N–7491 Trondheim\\
  Elena Celledoni: \texttt{elena.celledoni@ntnu.no}; Sølve Eidnes: \texttt{solve.eidnes@ntnu.no}; Brynjulf Owren: \texttt{brynjulf.owren@ntnu.no}; Torbjørn Ringholm: \texttt{torbjorn.ringholm@ntnu.no}\\
  This work was supported by the European Union's Horizon 2020 research and innovation programme under the Marie Skłodowska-Curie grant agreement No. 691070.}}

\begin{document}
\maketitle

\begin{abstract}
The energy preserving discrete gradient methods are generalized to finite-dimensional Riemannian manifolds by definition of a discrete approximation to the Riemannian gradient, a retraction, and a coordinate center function. The resulting schemes are intrinsic and do not depend on a particular choice of coordinates, nor on embedding of the manifold in a Euclidean space. Generalizations of well-known discrete gradient methods, such as the average vector field method and the Itoh--Abe method are obtained.
It is shown how methods of higher order can be constructed via a collocation-like approach. Local and global error bounds are derived in terms of the Riemannian distance function and the Levi-Civita connection. Some numerical results on spin system problems are presented.
 

\textbf{Classification}: 37K05, 53B99, 65L05, 82-08
\end{abstract}

\section{Introduction}
A first integral of an ordinary differential equation (ODE) is a scalar-valued function on the phase space of the ODE that is preserved along solutions. The potential benefit of using numerical methods that preserve one or more such invariants is well-documented, and several energy-preserving methods have been developed in recent years. Among these are the discrete gradient methods, which were introduced for use in Euclidean spaces in \cite{Gonzalez96}, see also \cite{McLachlan99}. These methods are based on the idea of expressing the ODE using a skew-symmetric operator and the gradient of the first integral, and then creating a discrete counterpart to this in such a way that the numerical scheme preserves the energy.

For manifolds in general, one can use the same schemes expressed in local coordinates. A drawback is that the numerical approximation will typically depend on the particular choice of coordinates and also on the strategy used for transition between coordinate charts. Another alternative is to use a global embedding of the manifold into a larger Euclidean space, but then it typically happens that the numerical solution deviates from the manifold. Even if the situation can be amended by using projection, it may not be desirable that the computed approximation depends on the particular embedding chosen.
Crouch and Grossmann \cite{Crouch93} and Munthe-Kaas \cite{Munthe95, Munthe98} introduced different ways of extending existing Runge--Kutta methods to a large class of differentiable manifolds. Both these approaches are generally classified as Lie group integrators, see \cite{Iserles00} or the more recent \cite{Celledoni14} for a survey of this class of methods. They can also both be formulated abstractly by means of a post-Lie structure which consists of a Lie algebra with a flat connection of constant torsion, see e.g. \cite{munthe-kaas13opl}. In the present paper we shall state the methods in a slightly different context, using the notion of a Riemannian manifold. It is then natural to make use of the Levi-Civita connection, which in contrast to the post-Lie setting is torsion-free, and which in general has a non-zero curvature. For our purposes it is also an advantage that the Riemannian metric provides an intrinsic definition of the gradient. Taking an approach more in line with this, Leimkuhler and Patrick \cite{Leimkuhler96} considered mechanical systems on the cotangent bundle of a Riemannian manifold and succeeded in generalising the classical leap-frog scheme to a symplectic integrator on Riemannian manifolds.

Some classical numerical methods in Euclidean spaces preserve certain classes of invariants; for instance, symplectic Runge--Kutta methods preserve all quadratic invariants. This can be useful when there is a natural way of embedding a manifold into a linear space by using constraints that are expressed by means of such invariants. An example is the 2-sphere which can be embedded in $\mathbb{R}^3$ by adding the constraint that these vectors should have unit length. The classical midpoint rule will automatically ensure that the numerical approximations remain on the sphere as it preserves all quadratic invariants. In general, however, the invariants preserved by these methods are expressed in terms of coordinates. Hence the preservation property of the method may be  
lost under coordinate changes if the invariant is no longer quadratic.
In \cite{Celledoni13}, a generalization of the discrete gradient method to differential equations on Lie groups and a broad class of manifolds was presented. Here we develop this further by introducing a Riemannian structure that can be used to provide an intrinsic definition of the gradient as well as a means to measure numerical errors.

The structure of this paper is as follows: In section 2, we formulate the problem to be solved and introduce discrete Riemannian gradient methods, as well as presenting some particular examples with special attention to a generalization of the Itoh--Abe discrete gradient. We also briefly discuss the Euclidean setting as a special choice of manifold and show how the standard discrete gradient methods are recovered in this case.
In the third section, we consider higher order energy preserving methods based on generalization of a collocation strategy introduced by Hairer  \cite{Hairer09} to Riemannian manifolds. We present some error analysis in 
section~4, and show numerical results in section~5, where the methods are applied to spin system problems.

\section{Energy preservation on Riemannian manifolds}
Consider an initial value problem on the finite-dimensional Riemannian manifold $(M,g)$,
\begin{align}
\dot{\uv} = F(\uv),\quad \uv(0)=\uv^0\in M.
\label{eq:ODE}
\end{align}
We denote by $\mathcal{F}(M)$ the space of smooth functions on $M$. The set of smooth vector fields and differential one-forms are 
denoted $\Gamma(TM)$ and $\Gamma(T^*M)$ respectively, and for the duality pairing between these two spaces we use the
angle brackets $\langle\cdot,\cdot\rangle$.

A first integral associated to a vector field $F\in\Gamma(TM)$ is a function $H \in\mathcal{F}(M)$ such that
$\langle \mathrm{d}H,F\rangle$ vanishes identically on $M$. First integrals are preserved along solutions of (\ref{eq:ODE}), 
\begin{align*}
\dfrac{\mathrm{d}}{\mathrm{d}t} H(\uv(t)) = \left\langle \mathrm{d}H(\uv(t)),\dot{\uv}(t) \right\rangle = \left\langle \mathrm{d}H(\uv(t)),F(\uv(t)) \right\rangle = 0.
\end{align*}

\subsection{Preliminaries}
The fact that a vector field $F$ has a first integral $H$ is closely related to the existence of a tensor field 
$\Omega\in\Gamma(TM\otimes T^*M)=: \Gamma(\mathcal{T}_1^1M)$, skew-symmetric with respect to the metric $g$, such that
\begin{align}
F(\uv) = \Omega(\uv) \, \grad H(\uv),
\label{eq:ODE_bivectorform2}
\end{align}
where $\grad H\in\Gamma(TM)$ is the Riemannian gradient, the unique vector field satisfying $\langle \mathrm{d}H, \cdot\rangle=g(\grad H,\cdot)$.
Any ODE \eqref{eq:ODE} where $F$ is of this form preserves $H$, since
\begin{align*}
\frac{\mathrm{d}}{\mathrm{d}t}H(\uv) &= \left\langle \mathrm{d}H(\uv), \dot{\uv} \right\rangle
= \left\langle \mathrm{d}H(\uv),  \Omega \, \grad H(\uv)\right\rangle
= g(\grad H(\uv),  \Omega \, \grad H(\uv))=0.
\end{align*}
A converse result is detailed in the following proposition.
\begin{prop}
Any system (\ref{eq:ODE}) with a first integral $H$ can be written with an $F$ of the form (\ref{eq:ODE_bivectorform2}).  
The skew tensor field $\Omega$ can be chosen so as to be bounded near every nondegenerate critical point of $H$.
\end{prop}
\begin{proof}
Similar to the proof of Proposition $2.1$ in \cite{McLachlan99}, we can write an explicit expression for a possible choice of $\Omega$,
\begin{align}
\Omega y &= \frac{g(\grad H,y)\,F-g(F,y)\,\grad H}{g(\grad H,\grad H)}.
\label{eq:bivector}
\end{align}
Clearly, $g(y,\Omega y)=0$ for all $y$.
Since $H$ is a first integral, $g(F,\grad H)=\langle \mathrm{d}H,F\rangle=0$, so $\Omega\,\grad H=F$. For a proof that $\Omega$ is bounded near nondegenerate critical points, see \cite{McLachlan99}.
\end{proof}
\noindent In fact, such a tensor field $\Omega$ often arises naturally from a two-form $\omega$ through 
 $\Omega y = \omega(\cdot,y)^\sharp$. A well-known example is when $\omega$ is a symplectic two-form. Note that $\Omega$ is not necessarily unique.

Retractions, viewed as maps from $TM$ to $M$, will play an important role in the methods we discuss here. 
Their formal definition can be found e.g. in \cite{Adler02}:
\begin{definition}
Let $\phi$ be a smooth map $\phi: TM \rightarrow M$ and let $\phi_\p$ denote the restriction of $\phi$ to $T_\p M$, with $0_\p$ being the zero-vector in $T_\p M$. Then $\phi$ is a \textit{retraction} if it satisfies the conditions
\begin{enumerate}
\item $\phi_\p$ is defined in an open ball $B_{r_\p}(0_\p) \subset T_\p M$ of radius $r_\p$ about $0_\p$,
\item $\phi_\p(\x) = \p$ \text{ if and only if } $\x=0_\p$,
\item $\dd\phi_\p\big|_{0_\p} = \mathrm{Id}_{T_\p M}$.
\end{enumerate}
\end{definition}
A generic example of a retraction on $(M,g)$ is obtained via the Riemannian exponential, setting $\phi_\p(\x)=\exp_\p(\x)$, i.e. following along the geodesic emanating from $\p$ in the direction $\x$.

\subsection{The discrete Riemannian gradient method}

We adapt the discrete gradients in Euclidean space to discrete Riemannian gradients (DRG) on $(M,g)$ by means of a retraction map $\phi$ and a center point function $c$.

\begin{definition}
A discrete Riemannian gradient is a triple $(\overline\grad, \phi, c)$\footnote{To avoid cluttered notation we will just write  $\overline{\grad}$ for the triple $(\overline\grad, \phi, c)$ in the sequel.}
 where
\begin{enumerate}
\item $c: M \times M \rightarrow M$ is a continuous map such that $c(\uv,\uv)=\uv$ for all $\uv \in M$,
\item $\overline{\grad}: \mathcal{F}(M)\rightarrow \Gamma(c^*TM)$,
\item $\phi: TM \rightarrow M$ is a retraction,
\end{enumerate}
such that for all $H\in\mathcal{F}(M)$, $u\in M$, $v\in M$, $c=c(u,v)\in M$,
\begin{align}
H(\vv) - H(\uv) &= g ( \overline{\grad}H(\uv,\vv),\phi_c^{-1}(\vv) - \phi_c^{-1}(\uv) ), \label{eq:discdiff1}\\
\overline{\grad}H(\uv,\uv) &=  \grad H(\uv). \label{eq:discdiff2}
\end{align}
\end{definition}
The DRG $\overline{\grad} H$ is a continuous section of the pullback bundle $c^*TM$, meaning that
${\pi\circ \overline{\grad} H = c}$, where $\pi:TM\rightarrow M$ is the natural projection.
We also need to define an approximation to be used for the tensor field $\Omega\in \Gamma(\mathcal{T}_1^1M)$. To this end  we let $\overline{\Omega}\in \Gamma(c^*\mathcal{T}_1^1M)$ be a continuous skew-symmetric tensor field such that
$$
\overline{\Omega}(\uv,\uv)=\Omega(\uv)\quad \forall\uv\in M.
$$
Inspired by \cite{Celledoni13,Celledoni18}, we propose the scheme
\begin{align}
\uv^{\kk+1} &= \phi_{c^k}(W(\uv^\kk,\uv^{\kk+1})) \label{eq:ddscheme1},\quad c^k=c(\uv^k,\uv^{k+1})\\
 W(\uv^\kk,\uv^{\kk+1}) &= \phi_{c^k}^{-1}(\uv^\kk) + h \, \overline{\Omega}(\uv^\kk,\uv^{\kk+1}) \, \overline{\grad} H(\uv^\kk,\uv^{\kk+1}),
 \label{eq:ddscheme2}
\end{align}
where $h$ is the step size. The scheme (\ref{eq:ddscheme1})--(\ref{eq:ddscheme2}) preserves the invariant $H$, since
\begin{align*}
H(\uv^{\kk+1}) - H(\uv^\kk) &= g ( \overline{\grad}H(\uv^\kk,\uv^{\kk+1}),\phi_{c^k}^{-1}(\uv^{\kk+1}) - \phi_{c^k}^{-1}(\uv^\kk) )\\
&= g ( \overline{\grad}H(\uv^\kk,\uv^{\kk+1}), h \, \overline{\Omega}(\uv^\kk,\uv^{\kk+1}) \, \overline{\grad} H(\uv^\kk,\uv^{\kk+1}))=0.
\end{align*}
Here and in the following we adopt the shorthand notation $c = c(u,v)$ as long as it is obvious what the arguments of $c$ are.

The Average Vector Field (AVF) method has been studied extensively in the literature; some early references are \cite{Harten83,McLachlan99,quispel08anc}. This is a discrete gradient method, and we propose a corresponding DRG satisfying  (\ref{eq:discdiff1})-(\ref{eq:discdiff2}) as follows:
\begin{equation}
\overline{\grad}_\text{AVF}H(\uv,\vv) = \int_0^1 (\dd_{\gamma_{\xi}}\phi_c)^\text{T} \, \grad H(\phi_c(\gamma_\xi)) \, \mathrm{d}\xi, \quad \gamma_\xi = (1-\xi)\phi_c^{-1}(\uv) + \xi \phi_c^{-1}(\vv),
\label{eq:AVF}
\end{equation}
where $(\dd_{x}\phi_c)^\text{T}: T_{\phi_c(x)}M\rightarrow T_xM$ is the unique operator satisfying
\begin{equation*}
g((\dd_{x}\phi_c)^\text{T}a,b)=g(a,\dd_{x}\phi_c\, b), \quad \forall x,b\in T_cM, \quad a\in T_{\phi_c(x)}M.
\end{equation*}
Furthermore, we have the generalization of Gonzalez' midpoint discrete gradient \cite{Gonzalez96},
\begin{equation}
\overline{\grad}_\text{MP} H(\uv,\vv) = \grad H(c(\uv,\vv)) + \frac{H(\vv)-H(\uv) - g ( \grad H(c(\uv,\vv)),\eta)}{g (\eta,\eta)}\eta
\label{eq:Gonzalez}
\end{equation}
where $\eta = \phi_c^{-1}(\vv) - \phi_c^{-1}(\uv)$.

Note that both these DRGs involve the gradient of the first integral.   This may be a disadvantage if $H$ is non-smooth or if its gradient is expensive to compute.
 Also, the implicit nature of the schemes requires the solution of an $n$-dimensional nonlinear system of equations at each time step. An alternative is to consider the Itoh--Abe discrete gradient \cite{Itoh88}, also called the coordinate increment discrete gradient \cite{McLachlan99}, which in certain cases requires only the solution of $n$ decoupled scalar equations. We now present a generalization of the Itoh--Abe discrete gradient to finite-dimensional Riemannian manifolds.

\subsection{Itoh--Abe discrete Riemannian gradient}

\begin{definition}
For any tangent space $T_c M$ one can choose a basis $\{E_1,...,E_n \}$ composed of tangent vectors $E_i$, $i=1,...,n$, orthonormal with respect to the Riemannian metric $g$.
Then, given $\uv, \vv \in M$, there exists a unique $\left\lbrace \alpha_i \right\rbrace_{i=1}^n$ so that
\begin{align*}
\phi_c^{-1}(\vv) - \phi_c^{-1}(\uv) = \sum \limits_{i=1}^n \alpha_i E_i.
\end{align*}
The \textit{Itoh--Abe DRG} of the first integral $H$ is then given by
\begin{align}
\overline{\grad}_\text{IA} H(\uv,\vv)& = \sum \limits_{j=1}^{n} a_j E_j,
\label{eq:ItohAbe}
\end{align}
where
\begin{align*}
a_j &=
  \begin{cases}
    \dfrac{H(w_j) - H(w_{j-1})}{\alpha_j}       & \quad \text{if } \alpha_j \neq 0,\\
    g(\grad H(w_{j-1}), \dd \phi_c(\eta_{j-1}) E_j)  & \quad \text{if } \alpha_j = 0,
  \end{cases}\\
w_j &= \phi_c(\eta_j), \quad \eta_j = \phi_c^{-1}(\uv) + \displaystyle{\sum_{i=1}^j} \alpha_i E_i.
\end{align*}
\end{definition}
We refer to \cite{Celledoni18} for proof that this is indeed a DRG satisfying (\ref{eq:discdiff1})-(\ref{eq:discdiff2}).

\subsection{Euclidean setting} \label{subseq:EuclideanSetting}
Let $M=V$ be an $\mathbb{R}$-linear space, and let $g$ be the Euclidean inner product, $g(x,y)=x^\mathrm{T} y$.
The operator $\Omega$ is a solution dependent skew-symmetric $n \times n$ matrix $\Omega(\uv)$.  For any $u\in V$, we have 
$T_uV\equiv V$.
The retraction $\phi: V \rightarrow V$ is defined as $\phi_\p(\x) = \p + \x$, the Riemannian exponential on $V$, so that $\phi^{-1}_c(\vv) - \phi^{-1}_c(\uv) = \vv - \uv$. The gradient $\grad H$ is an $n$-vector whose $i$th component is $\frac{\partial H}{\partial \uv_i}$, and the definition of the discrete Riemannian gradient coincides with the standard discrete gradient, since \eqref{eq:discdiff1} now reads
\begin{equation*}
H(\vv)-H(\uv) = \overline{\grad}(\uv,\vv)^\mathrm{T} (\vv-\uv).
\end{equation*}
Furthermore, (\ref{eq:ddscheme1})-(\ref{eq:ddscheme2}) simply becomes the discrete gradient method introduced in \cite{Gonzalez96}, given by the scheme
\begin{align}
 \uv^{k+1}-\uv^k &= h\overline{\Omega}(\uv^k,\uv^{k+1}) \, \overline{\grad} H(\uv^k,\uv^{k+1}),
 \label{eq:dgschemeeuclid}
\end{align}
where $\overline{\Omega}$ is a skew-symmetric matrix approximating $\Omega$. 
Typical choices are $\overline{\Omega}(\uv^k,\uv^{k+1}) = \Omega(\uv^k)$, or $\overline{\Omega}(\uv^k,\uv^{k+1}) = \Omega((\uv^{k+1}+\uv^k)/2)$ if one seeks a symmetric method.

The DRGs (\ref{eq:AVF}) and (\ref{eq:Gonzalez}) become the standard AVF and midpoint discrete gradients in this case. For the Itoh--Abe DRG, the practical choice for the orthogonal basis would be the set of unit vectors, $\{e_1,...,e_n \}$, so  that $\alpha_i = \vv_i - \uv_i$, and we get (\ref{eq:ItohAbe}) with
\begin{align*}
a_j &=
  \begin{cases}
    \dfrac{H(w_j) - H(w_{j-1})}{\vv_j-\uv_j}       & \quad \text{if } \uv_j \neq \vv_j,\\
    \frac{\partial H}{\partial \uv_j}(w_{j-1})  & \quad \text{if } \uv_j = \vv_j,
  \end{cases}\\
w_j &= \sum\nolimits_{i=1}^j \vv_i e_i + \sum\nolimits_{i=j+1}^n \uv_i e_i,
\end{align*}
which is a reformulation of the Itoh--Abe discrete gradient as it is given in \cite{Itoh88}, \cite{McLachlan99} and the literature otherwise.

\section{Methods of higher order} \label{sec:collocation}
In the Euclidean setting, a strategy to obtain energy preserving methods of higher order was presented in \cite{brugnano10hbv} and later in \cite{Hairer09}, see also \cite{Cohen11}. This technique is generalized to a Lie group setting in \cite{Celledoni13}. We will here formulate these methods in the context of Riemannian manifolds.

\subsection{Energy-preserving collocation-like methods on Riemannian manifolds}
Let $\cc_1,...,\cc_s$ be distinct real numbers. Consider the Lagrange basis polynomials,
\begin{equation} \label{eq:elldef}
l_i(\xi) = \prod_{j=1,j\neq i}^s \frac{\xi-\cc_j}{\cc_i-\cc_j},\quad\mathrm{and}\,\, \mathrm{let}\quad b_i := \int_0^1 l_i(\xi) \,d\xi.
\end{equation}
We assume that $\cc_1,\ldots,\cc_s$ are such that $b_i \neq 0$ for all $i$. 
A step of the energy-preserving collocation-like method, starting at $u^0 \in M$, is defined via a polynomial ${\sigma: \mathbb{R} \rightarrow T_cM}$ of degree $s$ satisfying
\begin{align}
&\sigma(0) = \phi_\ce^{-1}(\uv^0), \label{eq:collman1} \\ 
&\frac{d}{d\xi}\sigma(\xi h)\Big\rvert_{\xi=\cc_j} = \dd_{U_j}\phi_\ce^{-1}\left(\Omega_j \grad_j H \right), \quad U_j:=\phi_\ce\left(\sigma(\cc_j h)\right) \label{eq:collman2}\\
&u^1:=\phi_\ce\left(\sigma(h)\right), \label{eq:collman3}
\end{align}
where 
\begin{equation*}
\grad_j H := \int_0^1 \frac{l_j(\xi)}{b_j} \left(\dd_{U_j}\phi_\ce^{-1}\right)^\text{T} (\dd_{\sigma(\xi h)}\phi_\ce)^\text{T} \, \grad H\left(\phi_\ce(\sigma(\xi h))\right) \, \mathrm{d}\xi, \quad \text{and} \quad 
\Omega_j := \Omega({U_j}).
\end{equation*}
Notice that with $s=1$ and independently on the choice of $\cc_1$, we reproduce the DRG method \eqref{eq:ddscheme1}-\eqref{eq:ddscheme2} with the AVF DRG \eqref{eq:AVF}.

Using Lagrange interpolation and \eqref{eq:collman2}, the derivative of $\sigma(\xi h)$ at every point $\xi h$ is
\begin{equation}
\frac{d}{d\xi}\sigma(\xi h)= \sum_{j=1}^s l_j(\xi) \dd_{U_j}\phi_\ce^{-1}\left(\Omega_j \grad_j H \right),
\label{eq:colpolman}
\end{equation}
from which by integrating we get
\begin{equation*}
\sigma(\tau h)= \phi_\ce^{-1}(\uv_0) + h \sum_{j=1}^s \int_0^{\tau} l_j(\xi) \, \mathrm{d}\xi \, \dd_{U_j}\phi_\ce^{-1}\left(\Omega_j \grad_j H \right).
\end{equation*}

The defined method is energy preserving, which we see by using 
\begin{equation*}
\frac{\mathrm{d}}{\mathrm{d}\xi}\left( \phi_\ce(\sigma(\xi h))\right) = \, \dd_{\sigma(\xi h)} \phi_\ce \left( \frac{\mathrm{d}}{\mathrm{d}\xi} \sigma(\xi h)\right),
\end{equation*}
and (\ref{eq:colpolman}) to get
\begin{align*}
H(\uv^1&) - H(\uv^0) = \int_{0}^{1} g\left( \grad H\left(\phi_\ce(\sigma(\xi h))\right), \, \frac{\mathrm{d}}{\mathrm{d}\xi} \phi_\ce(\sigma(\xi h)) \right) \, \mathrm{d}\xi\\
&= \int_{0}^{1} g\left( \grad H\left(\phi_\ce(\sigma(\xi h))\right), \, \dd_{\sigma(\xi h)} \phi_\ce \left( \sum_{j=1}^s l_j(\xi) \, \dd_{U_j}\phi_\ce^{-1}\left(\Omega_j \grad_j H \right) \right) \right) \, \mathrm{d}\xi \\
&= \int_{0}^{1} g\left( \left(\dd_{\sigma(\xi h)} \phi_\ce\right)^\text{T} \grad H\left(\phi_\ce(\sigma(\xi h))\right), \, \sum_{j=1}^s l_j(\xi) \, \dd_{U_j}\phi_\ce^{-1}\left(\Omega_j \grad_j H  \right) \right) \, \mathrm{d}\xi \\
&= \sum_{j=1}^s b_j g\left( \int_{0}^{1} \frac{l_j(\xi)}{b_j} \left(\dd_{U_j}\phi_\ce^{-1}\right)^\text{T} \left(\dd_{\sigma(\xi h)} \phi_\ce\right)^\text{T} \grad H\left(\phi_\ce(\sigma(\xi h))\right) \mathrm{d} \xi, \, \Omega_j \grad_j H \right) \\
&= \sum_{j=1}^s b_j g\left(  \grad_j H, \, \Omega_j \,\grad_j H \right) = 0,
\end{align*}
and hence repeated use of (\ref{eq:collman1})-(\ref{eq:collman3}) ensures
$H(\uv^k) = H(\uv^0)$ for all $k \in \mathbb{N}$.

\subsection{Higher order extensions of the Itoh--Abe DRG method}
From the Itoh--Abe DRG one  can get a new DRG, also satisfying \eqref{eq:discdiff1}, by
\begin{align}
\overline{\grad}_\text{SIA} H(\uv,\vv)& = \frac{1}{2}\left(\overline{\grad}_\text{IA} H(\uv,\vv)+\overline{\grad}_\text{IA} H(\vv, \uv)\right).
\label{eq:SymItohAbe}
\end{align}
We call this the \textit{symmetrized Itoh--Abe DRG}. Note that we need the base point $c$ to be the same in the evaluation of $\overline{\grad}_\text{IA} H(\uv,\vv)$ and $\overline{\grad}_\text{IA} H(\vv,\uv)$. When $c(\uv,\vv) = c(\vv,\uv)$ and $\overline{\Omega}_{(u,v)} = \overline{\Omega}_{(v,u)}$, we get a symmetric DRG method \eqref{eq:ddscheme1}-\eqref{eq:ddscheme2}, which is therefore of second order.

Alternatively, one can get a symmetric $2$-stage method by a composition of the Itoh--Abe DRG method and its adjoint. 
Furthermore, one can get energy preserving methods of any order using a composition strategy. 
To ensure symmetry of an $s$-stage composition method, one needs $c_i(\uv,\vv) = c_{s+1-i}(\vv,\uv)$ for different center points $c_i$ belonging to each stage and, similarly, $\overline{\Omega}_i(\uv,\vv) = \overline{\Omega}_{s+1-i}(\vv,\uv)$.

\section{Error analysis}
\subsection{Local error}
In this section, $\varphi_t(u)$ is the $t$-flow of the ODE vector field $F$. 
The  most standard discrete gradient methods  have a low or moderate order of convergence, and that is also the case for the DRG methods unless special care is taken in designing $\overline{\Omega}$ and $\overline{\grad} H$. 
We shall not pursue this approach here, but refer to the collocation-like methods if high order of accuracy is required. 
We shall see, however, that the methods designed here are consistent and can be made symmetric.
Analysis of the local error can be done in local coordinates, assuming that the step size is always chosen sufficiently small, so that within a fixed step, 
$\uv^k, \uv^{k+1}, c(u^k,u^{k+1})$ and the exact local solution $\uv(t_{k+1})$ all belong to the same given coordinate chart. From the definition
\eqref{eq:ddscheme1}-\eqref{eq:ddscheme2} it follows immediately that the representation of $\uv^{k+1}(h)$ satisfies $\uv^{k+1}(0)=u^k$ and
$\frac{d}{dh} u^{k+1}(0)=F(u^k)$. Then by equivalence of local coordinate norms and the Riemannian distance, we may conclude that the local error in DRG methods satisfies 
$$
     d(u^{k+1},\varphi_h(u^k)) \leq C h^2.
$$
Similar to what was also observed in \cite{Celledoni13}, the DRG methods \eqref{eq:ddscheme1}-\eqref{eq:ddscheme2} are symmetric whenever
$\overline{\grad} H(u,v)=\overline{\grad} H(v,u)$,  $\overline{\Omega}(u,v)=\overline{\Omega}(v,u)$, and $c(u,v)=c(v,u)$ for all $u,v\in M$. In that case we obtain an error bound for the local error of the form
$d(u^{k+1},\varphi_h(u^k)) \leq C h^3$.

The collocation-like methods of section~\ref{sec:collocation} have associated nodes $\{\cc_i\}_{i=1}^s$ and weights $\{b_i\}_{i=1}^s$ defined by
\eqref{eq:elldef}. The order of the local error depends on the accuracy of the underlying quadrature formula given by these nodes and weights.
The following result is a simple consequence of Theorem~4.3 in \cite{Cohen11}.

\begin{theo}
Let $\psi_h$ be
the method defined by \eqref{eq:collman1}-\eqref{eq:collman3}. The order of the local error is at least
$$
p=\min(r,2r-2s+2)
$$
where $r$ is the largest integer such that $\sum_{i=1}^s b_i c_i^{q-1}=\frac{1}{q}$ for all $1\leq q\leq r$. This means that there are positive constants $C$
and $h_0$ such that 
$$
     d(\psi_h(\uv),\varphi_h(\uv))\leq C\,h^{p+1}\quad \text{for}\ h<h_0,\ u\in M.
$$
\end{theo}
\begin{proof}
Choose $h$ small enough such that 
the solution can be represented in the form $u(h\xi)=\phi_c(\gamma(\xi h)),\xi \in[0,1],$ and consider 
the corresponding differential equation for $\gamma$ in $T_{c}M$:
\begin{equation}
\label{localequation}
\frac{d}{dt}\gamma(t)= \left(\phi_c^* F\right)(\gamma(t))=
\left( T_{\gamma (t)}\phi_{c}\right)^{-1}  \Omega \, \grad H \left(\phi_{c}(\gamma(t) ) \right).
\end{equation}
Notice that $\left( T_{\gamma}\phi_{c}\right)^{-1}=T_{U}\phi_{c}^{-1}$ where $U=\phi_c\circ \gamma$ and $T_{U(t)}\phi_{c}^{-1}:T_{U(t)}M\rightarrow T_cM$ for every $t$. We obtain
\begin{equation}
\label{localequationIntermidiate}
\frac{d}{d t}\gamma(t)=T_{U(t)}\phi_{c}^{-1} \Omega \,\left(T_{U(t)}\phi_{c}^{-1}  \right)^{\text{T}} \left(T_{\gamma(t)}\phi_{c}\right)^{\text{T}}\grad H \left(\phi_{c}(\gamma(t) ) \right) .
\end{equation}
Considering the Hamiltonian $\widetilde{H}:T_cM\rightarrow \mathbb{R}$, $\widetilde{H}(\gamma):=\phi_c^*H(\gamma)= H\circ\phi_c(\gamma)$, we can then rewrite \eqref{localequation} in the form 
\begin{equation}
\label{localequation1}
\frac{d}{dt}\gamma(t)=\widetilde{\Omega}(\gamma)\, \grad \widetilde{H} (\gamma),\quad \widetilde{\Omega}(\gamma):=
T_{U(t)}\phi_{c}^{-1} \Omega \,\left(T_{U(t)}\phi_{c}^{-1}  \right)^{\text{T}} ,
\end{equation}
where we have used that $\grad\widetilde{H}=T_{\gamma(t)}\phi_{c}^{\text{T}}\grad H\,(\phi_{c}(\gamma(t)) )$, which is now a gradient on the linear space $T_cM$ with respect to the metric inherited from $M$, $g|_c$. 
Locally in a neighborhood of $c$, \eqref{eq:collman1}-\eqref{eq:collman3} applied to \eqref{localequation1} coincides with the methods of  Cohen and Hairer, 
and therefore the order result \cite[Thm 4.3]{Cohen11} can be applied. Since the Riemannian distance $d(\cdot,\cdot)$ and any norm in local coordinates are equivalent, the result follows.
\end{proof}

\subsection{Global error}
We prove the following result for the global error in DRG methods.
\begin{theo}
Let $\uv(t)$ be the exact solution to (\ref{eq:ODE})
where $F$ is a complete vector field on a connected Riemannian manifold $(M,g)$ with flow $\uv(t)=\varphi_t(\uv^0)$.
Let $\psi_\h$ represent a numerical method $\uv^{\kk+1}=\psi_\h(\uv^\kk)$ whose local error can be bounded for some $p \in \mathbb{N}$ as
$$
   d(\psi_\h(\uv),\varphi_\h(\uv)) \leq C \h^{p+1}\quad\text{for all}\ \uv\in M.
$$
Suppose there is a constant $L$ such that
$$
 \| \nabla F \|_g \leq L,
$$
where $\nabla$ is the Levi-Civita connection and $\|\cdot\|_g$ is the operator norm with respect to the metric $g$.
Then the global error is bounded as
$$
d\left(\uv(\kk\h),\uv^\kk\right) \leq \frac{C}{L}(\e^{\kk\h L}-1) \h^p\quad\text{for all}\ \kk>0.
$$
\end{theo}
\begin{proof}
Denoting the global error as
$
e^{\kk} 	:= d(\uv(\kk\h),\uv^\kk),
$
the triangle inequality yields
$$
e^{\kk+1} \leq d\left(\varphi_\h(\uv(\kk\h)\right),\varphi_\h(\uv^\kk))+d\left(\varphi_\h(\uv^\kk),\psi_\h(\uv^\kk)\right).
$$
The first term is the error at $n\h$ propagated over one step, the second term is the local error.
For the first term, we find via a Gr\"onwall type inequality of \cite{kunzinger06gge},
$$
 d\left(\varphi_\h(\uv(\kk\h)),\varphi_\h(\uv^\kk)\right) \leq \e^{\h L} d\left(\uv(\kk\h),\uv^\kk\right)=\e^{\h L}e^\kk.
$$
Using the local error estimate for the second term, we get the recursion
$$
       e^{\kk+1}  \leq \e^{\h L}e^\kk + C \h^{p+1} ,
$$
which yields
$$
  e^\kk \leq C\frac{\e^{\kk\h L}-1}{\e^{\h L}-1} \h^{p+1} \leq \frac{C}{L}(\e^{\kk\h L}-1) \h^p.
$$
\end{proof}
\textbf{\textit{Remark:}} Following Theorem 1.4 in \cite{kunzinger06gge}, the condition that $F$ is complete can be relaxed if $\varphi_t(u^0)$ and $\{u^k\}_{k\in \mathbb{N}}$ lie in a relatively compact submanifold $N$ of $M$ containing all the geodesics from $u^k$ to $\varphi_{kh}(u^0)$. This is the case if, for instance, $H$ has compact, geodesically convex sublevel sets, since both $\varphi_t(u^0)$ and $\{u^k\}_{k\in \mathbb{N}}$ are restricted to the level set $M_{H(u^0)} = \{ p \in M \, | \, H(p) = H(u^0) \}$ and hence lie in the sublevel set $N_{H(u^0)} = \{ p \in M \, | \, H(p) \leq H(u^0) \}$.


\section{Examples and numerical results}
We test our methods on two different variants of the classical spin system, whose solution evolves on the $d$-fold product of two-spheres, $(S^2)^d$,
\begin{equation}\label{spinsystem1}
\frac{\mathrm{d}\s_i}{\mathrm{d}t} = s_i \times \frac{\partial H}{\partial s_i}, \quad s_i \in S^2, \quad i=1,\ldots,d, \quad H \in \mathcal{F}\left(\left(S^2\right)^d\right).
\end{equation}
The Riemannian metric $g$ on $(S^2)^d$ restricts to the so-called round metric on each copy of the sphere. This metric coincides with the Euclidean inner product on the tangent planes of each of the spheres.

Geometric integrators for such systems are discussed widely in the literature, see e.g. \cite{Frank97, Lewis03, McLachlan14, McLachlan17} and references therein. 
We study one or more bodies whose orientation is represented by a vector $s_i$ of unit length in $\mathbb{R}^3$, so that $s_i$ lies on the manifold $M = S^2 = \left\{ s \in \mathbb{R}^{3} : \left\| s \right\| = 1 \right\}$.
Here and in what follows, $\lVert \cdot \rVert$ denotes the $2$-norm. 
Starting with $d=1$, our choice of retraction $\phi$ is given by its restriction to $\p$,
\begin{equation}
\phi_\p(\x) = \frac{\p+\x}{\left\lVert \p+\x \right\rVert},
\label{eq:retraction}
\end{equation}
with the inverse
\begin{equation*}
\phi_\p^{-1}(\uv) = \frac{\uv}{\p^\mathrm{T}\uv}-\p
\end{equation*}
defined when $\p^\mathrm{T}\uv > 0$. We note that $\p^\mathrm{T}\x=0$ for all $\x \in T_\p S^2$.
The tangent map of the retraction and its inverse are given by 
\begin{align}
\dd_\x\phi_\p = \frac{1}{\left\lVert \p+\x \right\rVert}\left(I - \frac{(\p+\x) \otimes (\p+\x)}{\left\lVert \p+\x \right\rVert^2}\right), \qquad \dd_\uv\phi_\p^{-1} = \frac{1}{\p^\mathrm{T}\uv}\left(I - \frac{\uv \otimes \p}{\p^\mathrm{T}\uv}\right),
\label{eq:tangentmaps}
\end{align}
where $\otimes $ denotes the outer product\footnote{ If $x$ and $y$ are in $\mathbb{R}^3$, $x\otimes y$ is the matrix-matrix product of $x$ taken as a $3\times 1$ matrix and $y$ taken as a $1\times 3$ matrix.} of the vectors. For $d>1$, we use the retraction defined by $\Phi_p(x) = (\phi_{p_1}(x_i),\dots,\phi_{p_d}(x_d))$, where each $\phi_{p_i}(x_i)$ is given by (\ref{eq:retraction}).

\subsection{Example 1: Perturbed spinning top}
We consider first a nonlinear perturbation of a spinning top, see \cite{McLachlan14}. This is a spin system with one spin $\s$.
Given the inertia tensor $\mathbb{I} = \text{diag}(\mathbb{I}_1,\mathbb{I}_2,\mathbb{I}_3)$, and denoting by $\s^2$ the component-wise square of $\s$,  we can define the Hamiltonian as
\begin{align*}
H(\s) = \frac{1}{2}(\mathbb{I}^{-1}\s)^\mathrm{T}(\s+\frac{2}{3}\s^2).
\end{align*}
The ODE system 
can be written in the form 
\begin{align*}
\dfrac{\mathrm{d}\s}{\mathrm{d}t} &= \Omega(\s) \, \grad H(\s), \quad \Omega(\s) = \hat{s},
\end{align*}
using the hat operator defined by $\hat{s} y = s \times y$.
We approximate this system numerically, testing the scheme (\ref{eq:ddscheme1})-(\ref{eq:ddscheme2}) with different discrete Riemannian gradients: the AVF (\ref{eq:AVF}), the midpoint (\ref{eq:Gonzalez}), the Itoh--Abe (\ref{eq:ItohAbe}) and its symmetrized version (\ref{eq:SymItohAbe}).
For the three symmetric methods, we have chosen $c(\s,\sh) = \frac{\s+\sh}{\lVert \s+\sh\rVert}$, so that $\phi_\ce^{-1}(\sh)=-\phi_\ce^{-1}(\s)$.
Using that $\grad H(\s) = \mathbb{I}^{-1} (\s+\s^2)$ and considering the transpose of $T_{\gamma_\xi}\phi_\ce$ from ({\ref{eq:tangentmaps}), the AVF DRG becomes
\begin{align*}
\overline{\grad}_\text{AVF} H(\s,\sh) &= \int_0^1 \frac{1}{\lVert l_{\xi}\rVert}\left(I - \frac{l_{\xi}\otimes l_{\xi}}{\lVert l_{\xi} \rVert^2}\right) \mathbb{I}^{-1}(\phi_\ce(\gamma_\xi)+\phi_\ce(\gamma_\xi)^2) \, \mathrm{d}\xi\\
&= \int_0^1 \frac{1}{\lVert l_{\xi} \rVert}\left(\mathbb{I}^{-1}\left(\phi_\ce(\gamma_\xi)+\phi_\ce(\gamma_\xi)^2 \right) - \phi_\ce(\gamma_\xi)^\mathrm{T}\mathbb{I}^{-1}\left(\phi_\ce(\gamma_\xi)+\phi_\ce(\gamma_\xi)^2\right)\phi_\ce(\gamma_\xi)\right) \, \mathrm{d}\xi,
\end{align*}
with $\gamma_\xi = (1-\xi)\phi_\ce^{-1}(\s) + \xi \phi_\ce^{-1}(\sh) = (1-2 \xi)\phi_\ce^{-1}(\s)$ and $l_{\xi} = c + \gamma_{\xi}$.
Similarly, the midpoint DRG becomes
\begin{align*}
\overline{\grad}_\text{MP} H(\s,\sh) &= 
\frac{1}{\lVert \s+\sh \rVert}\left(\mathbb{I}^{-1}\left(\s+\sh+\frac{2}{3}\left(\s^2 + \s \sh + \sh^2 \right)\right) + \frac{\frac{1}{2}\lVert \s+\sh \rVert^2 - 2}{\lVert \sh-\s\rVert^2} \left(H(\sh)-H(\s)\right)(\sh-\s)\right),
\end{align*}
where we have used that $g(\s,\s) = \s^\mathrm{T}\s = 1$ for all $\s \in S^2$.
To obtain the basis of $T_c M$ for the definition of the Itoh--Abe DRG, we have used the singular-value decomposition. 
For the first order scheme, noting that $\phi_\s^{-1}(\s)=0$, we choose  $c(\s,\sh) = \s$, and get $\alpha_j = \phi_{\s}(\sh)^\mathrm{T} E_j$, for $j=1,2$. Then the DRG (\ref{eq:ItohAbe}) can be written as 
\begin{align}
\overline{\grad}_\text{IA}H(\s,\sh) =& \frac{H\left(\phi_\s \left( \phi_{\s}^{-1}(\sh)^\mathrm{T} E_1  E_1\right)\right)- H(\s)}{ \phi_{\s}^{-1}(\sh)^\mathrm{T} E_1 } E_1 
+ \frac{H(\sh) - H\left(\phi_\s\left( \phi^{-1}_{\s}(\sh)^\mathrm{T} E_1 E_1\right)\right)}{ \phi_{\s}^{-1}(\sh)^\mathrm{T} E_2 } E_2.
\label{eq:IAexample}
\end{align}

We solve the same problem using the 4th, 6th and 8th order variants of the collocation-like scheme (\ref{eq:collman1})-(\ref{eq:collman3}). Choosing in the 4th order case the Gaussian nodes $\cc_{1,2} = \frac{1}{2}\mp\frac{\sqrt{3}}{6}$ as collocation points and setting $c(\s,\sh) = \s$, we get the nonlinear system
\begin{align*}
S_1 &= h \, \phi_{\s_0}\left( \frac{1}{2} \, T_{S_1}\phi_{\s_0}^{-1}\left(\Omega_1 \, \grad_1 H \right) + \left(\frac{1}{2}-\frac{\sqrt{3}}{3}\right) \, T_{S_2}\phi_{\s_0}^{-1}\left(\Omega_2 \, \grad_2 H \right) \right),\\
S_2 &= h \, \phi_{\s_0}\left( \left(\frac{1}{2}+\frac{\sqrt{3}}{3}\right) \, T_{S_1}\phi_{\s_0}^{-1}\left(\Omega_1 \, \grad_1 H \right) + \frac{1}{2} \, T_{S_2}\phi_{\s_0}^{-1}\left(\Omega_2 \, \grad_2 H \right) \right),\\
s_1 &= h \, \phi_{\s_0}\left( T_{S_1}\phi_{\s_0}^{-1}\left(\Omega_1 \, \grad_1 H\right) +  T_{S_2}\phi_{\s_0}^{-1}\left(\Omega_2 \, \grad_2 H \right) \right),
\end{align*}
where 
\begin{align*}
\sigma(\xi h) = \left( \left(3+2\sqrt{3}\right)\phi_{\s_0}^{-1}(S_{1}) + \left(3-2\sqrt{3} \right)\phi_{\s_0}^{-1}(S_{2})\right)  \xi + \left(3\left(\sqrt{3}-1 \right)\phi_{\s_0}^{-1}(S_{2}) -3 \left(1+\sqrt{3} \right) \phi_{\s_0}^{-1}(S_{1}) \right) \xi^2
\end{align*}
and we use the transposes of (\ref{eq:tangentmaps}) and $\grad H (s) = \mathbb{I}^{-1} (\s+\s^2)$ in the evaluation of $\grad_1 H$ and $\grad_2 H$. The 6th and 8th order schemes are derived in a similar manner, using the standard Gaussian nodes.

A second order scheme is derived by composing the Itoh--Abe DRG method with its adjoint, and a 4th order scheme is obtained by composing this method again with itself, as well as one by composition of the symmetrized Itoh--Abe DRG method with itself. In all stages of these composition methods, a symmetric $c(u,v)$ is used.

Plots confirming the order of all methods can be seen in Figure \ref{fig:orderplots}, where solutions using the different schemes are compared to a reference solution obtained using a very small step size.
See the left hand panel of Figure \ref{fig:sphere} for numerical confirmation that our methods do indeed preserve the energy to machine precision, while the implicit midpoint method does not. 
In the right hand panel of Figure \ref{fig:sphere}, the solution obtained by the Itoh--Abe DRG scheme with a step size $h=1$ is plotted together with a solution obtained using the symmetrized Itoh--Abe DRG method with a much smaller time step. We observe, as expected for a method that conserves both the energy and the angular momentum, that the solution stays on the trajectories of the exact solution, although not necessarily at the right place on the trajectory at any given time.

\begin{figure}
        \centering
        \subfloat{
                \centering
                \includegraphics[width=0.46\textwidth]{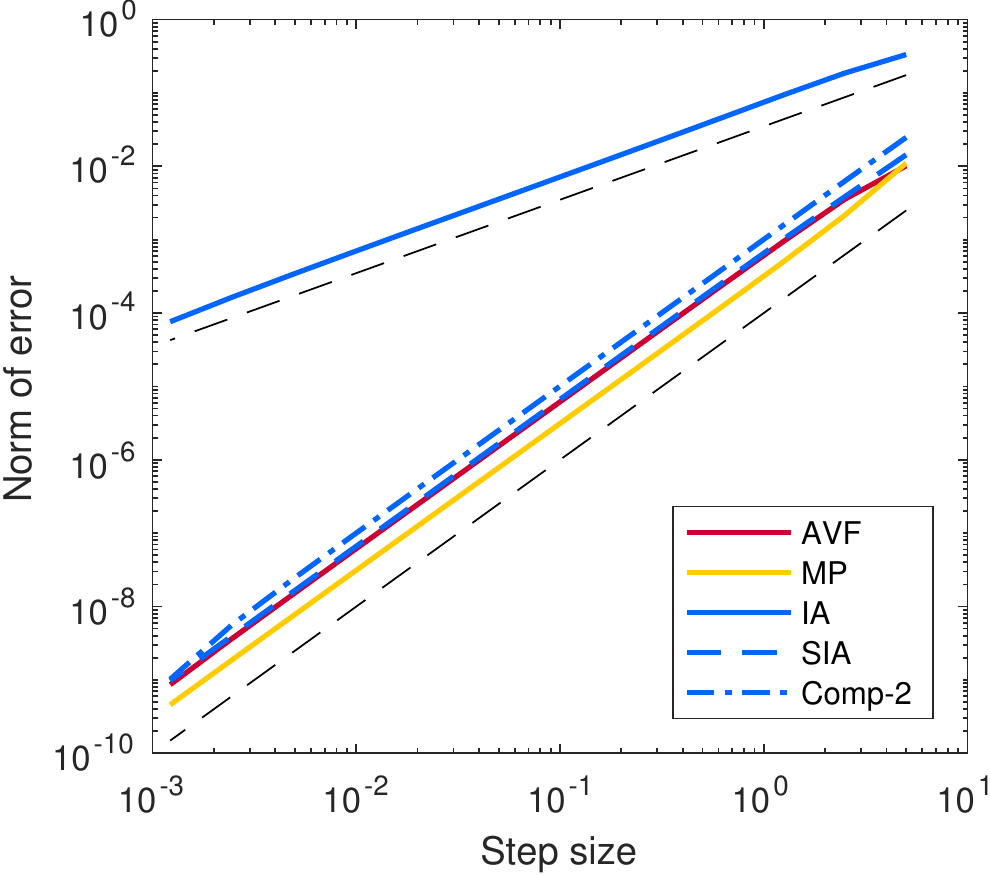}
        }\quad
        \subfloat{
                \centering
                \includegraphics[width=0.46\textwidth]{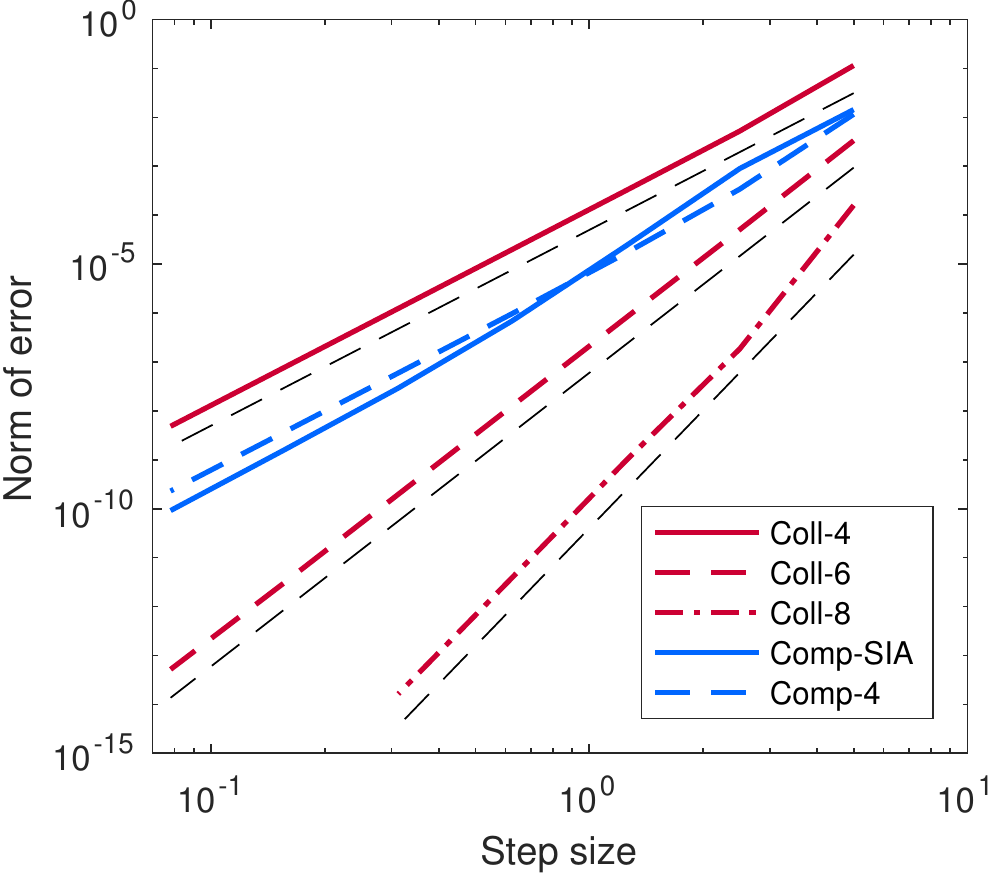}
        }
        \caption{Error norm at $t=10$ for the perturbed spinning top problem solved with different schemes, plotted with black, dashed reference lines of order 1, 2, 4, 6 and 8. Initial condition ${\s = (-1,-1,1)/\sqrt{3}}$ and $\mathbb{I} = \text{diag}(1,2,4)$. \textit{Left:} The AVF, midpoint (MP), Itoh--Abe (IA) and symmetrized Itoh--Abe (SIA) DRGs and a 3-stage composition of the IA DRG scheme (Comp-2). \textit{Right:} Collocation-type schemes of order 4, 6 and 8, a 3-stage composition of the SIA DRG scheme (Comp-SIA), and a 6-stage composition of the IA DRG scheme (Comp-4).}
        \label{fig:orderplots}
\end{figure}

\begin{figure}
        \centering
        \subfloat{
                \centering
                \includegraphics[width=0.46\textwidth]{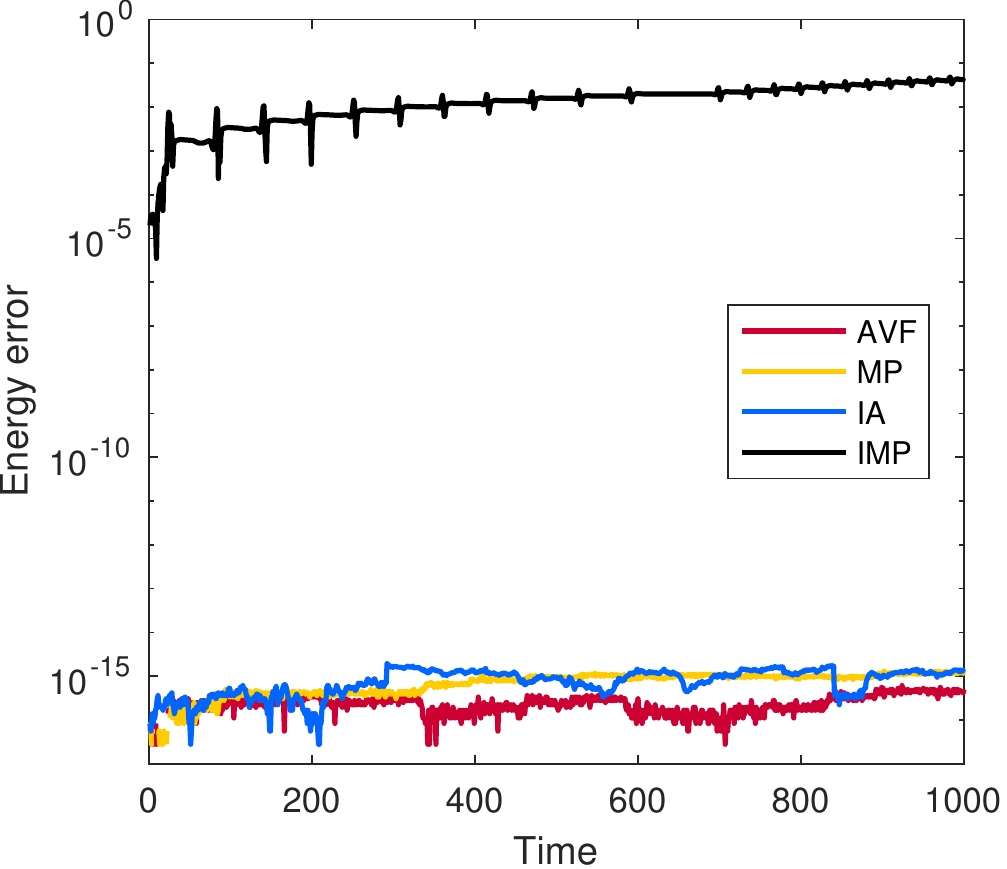}
        }\quad
        \subfloat{
                \centering
                \includegraphics[width=0.46\textwidth]{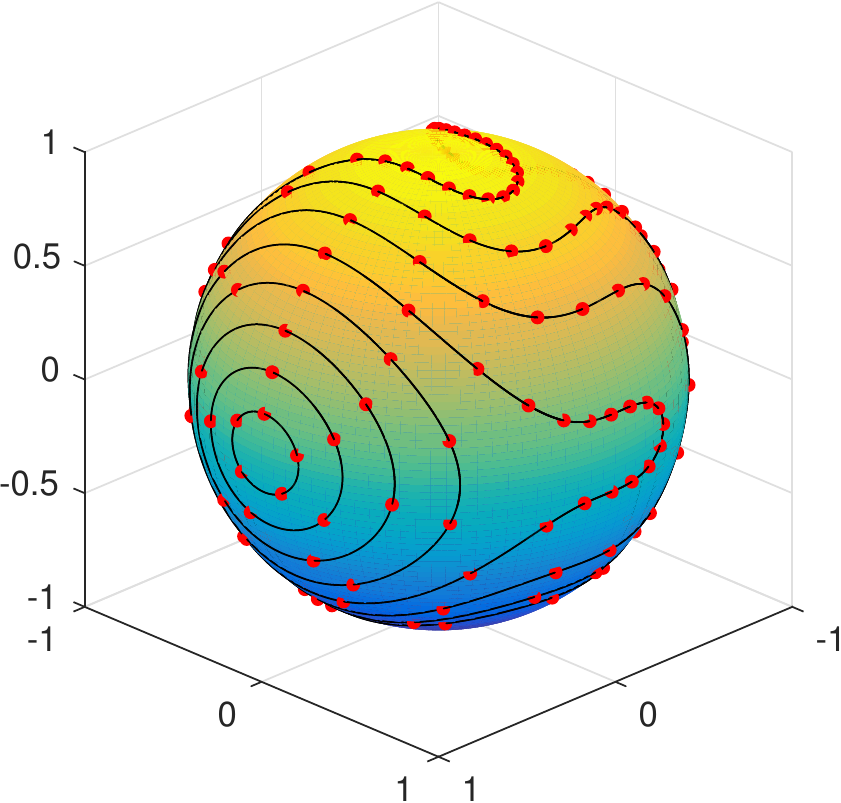}
        }
        \caption{\textit{Left:} Energy error with increasing time for the AVF, midpoint (MP) and Itoh--Abe (IA) DRG methods, as well as the implicit midpoint (IMP) method, with step size $h=1$, initial condition ${\s = (-1,-1,1)/\sqrt{3}}$ and $\mathbb{I} = \text{diag}(1,2,4)$. \textit{Right:} Curves of constant energy on the sphere, found by our method with different starting values. The black solid line is the solution using the symmetrized Itoh--Abe DRG method with step size $h=0.01$, while the red dots are the solutions obtained by the Itoh--Abe DRG method with step size $h=1$.}
        \label{fig:sphere}
\end{figure}

\subsection{Example 2: Heisenberg spin chain}
We now consider the Heisenberg spin chain of micromagnetics. This problem is considered in \cite{Frank97, McLachlan17}, where different geometric integrators are tested. Here, $\s \in \left(S^2 \right)^d$, and the Hamiltonian is
\begin{align}
H(\s) = \sum_{i=1}^d {s_i}^\mathrm{T}  s_{i-1},
\label{eq:spinchainhamiltonian}
\end{align}
with  $\s_0 = \s_d$ and $\s_{d+1} = \s_1$.
The system \eqref{spinsystem1} becomes, for this Hamiltonian, 
\begin{align*}
\dfrac{\mathrm{d}\s_i}{\mathrm{d}t} = \hat{s}_i \,\left(\s_{i-1}+\s_{i+1}\right), \quad i = 1,\ldots,d,
\end{align*}
and can be written in the block form
\begin{align}
\dfrac{\mathrm{d}\s}{\mathrm{d}t} = \Omega(\s) \, \grad H(s), \quad \mathrm{where} \quad \Omega(\s)  = \mathrm{diag}(\hat{s}_1,\dots , \hat{s}_d).
\label{eq:blockform}
\end{align}

For such a $d$-particle system, we may write the DRGs as
$$\overline{\grad} H(\s,\sh) = \left(\overline{\grad}^1 H(\s,\sh), \ldots, \overline{\grad}^d H(\s,\sh)\right),$$
where we note that $\overline{\grad}^i H(\s,\sh)$ is a discrete approximation to $\frac{\partial H}{\partial s_i}$. We thus get the AVF DRG defined by
\begin{align*}
\overline{\grad}^i_\mathrm{AVF} H(\s,\sh) =&\int_0^1 \! \frac{1}{\lVert l_{i,\xi}  \rVert}\left(I - \frac{l_{i,\xi} \otimes l_{i,\xi} }{\lVert l_{i,\xi} \rVert^2}\right) \left(\phi_{\ce_{i-1}}(\gamma_{i-1,\xi})+\phi_{\ce_{i+1}}(\gamma_{i+1,\xi})\right) \mathrm{d}\xi \\
=& \int_0^1 \! \frac{1}{\lVert l_{i,\xi}  \rVert}\left(\phi_{\ce_{i-1}}(\gamma_{i-1,\xi})+\phi_{\ce_{i+1}}(\gamma_{i+1,\xi}) - l_{i,\xi}^\mathrm{T} \left(\phi_{\ce_{i-1}}(\gamma_{i-1,\xi}) +\phi_{\ce_{i+1}}(\gamma_{i+1,\xi})\right)l_{i,\xi} \right)  \mathrm{d}\xi,
\end{align*}
with $\gamma_{i,\xi} = (1-2 \xi)\phi_{\ce_i}^{-1}(\s_i)$ and $l_{i,\xi} = \ce_i+\gamma_{i,\xi}$.
For the midpoint DRG we get
\begin{align*}
\overline{\grad}^i_\mathrm{MP} H(\s,\sh) = \ce_{i-1} + \ce_{i+1} + \frac{H(\sh) - H(\s) - (\grad H(\ce(\s,\sh))^\mathrm{T}\eta}{{\eta}^\mathrm{T}\eta}\eta_i,
\end{align*}
where $\eta = (\eta_1,\ldots,\eta_d)$ and $\eta_i = -2 \phi_{\ce_i}^{-1}(s_i).$ In the numerical experiments, however, we have used a small modification of this,
\begin{align*}
\overline{\grad}^i_\mathrm{MMP} H(\s,\sh) = \ce_{i-1} + \ce_{i+1} + \frac{{\sh_i}^\mathrm{T}\sh_{i-1} - {\s_{i}}^\mathrm{T}\s_{i-1} - (\ce_{i-1}+\ce_{i+1})^\mathrm{T}\eta_i}{{\eta_i}^\mathrm{T}\eta_i}\eta_i.
\end{align*}
This DRG, which does indeed satisfy (\ref{eq:discdiff1})-(\ref{eq:discdiff2}), leads to a more computationally efficient scheme than the original midpoint DRG.
Each $\overline{\grad}^i_\mathrm{IA} H(\s,\sh)$ in the Itoh--Abe DRG is found as in the previous example, by (\ref{eq:IAexample}). Higher order schemes are also derived in the same manner as before.

We test our schemes by comparing the numerical solutions with the exact solution
\begin{align*}
s_j(t) = (a \cos \theta_j + \widetilde{a} \sin \theta_j) \cos \phi + \bar{a} \sin \phi, \quad \theta_j = j p - 2(1-\cos p)\sin \phi,
\end{align*}
for a choice of constants $\phi, p \in \mathbb{R}$ and orthogonal unit vectors $a,\widetilde{a},\bar{a} \in \mathbb{R}^3$, see \cite{Frank97}. Order plots for the methods are provided in Figure \ref{fig:spin_orderplots}, using $d=5$, $\phi=\pi/3$, $p = 2 \pi/d$, $a = (1,2,-1)/\sqrt{6}$, $\widetilde{a} = (2,1,4)/\sqrt{21}$ and $\bar{a} = a \times \widetilde{a}.$ All schemes are shown to have the expected order.

\begin{figure}
        \centering
        \subfloat{
                \centering
                \includegraphics[width=0.46\textwidth]{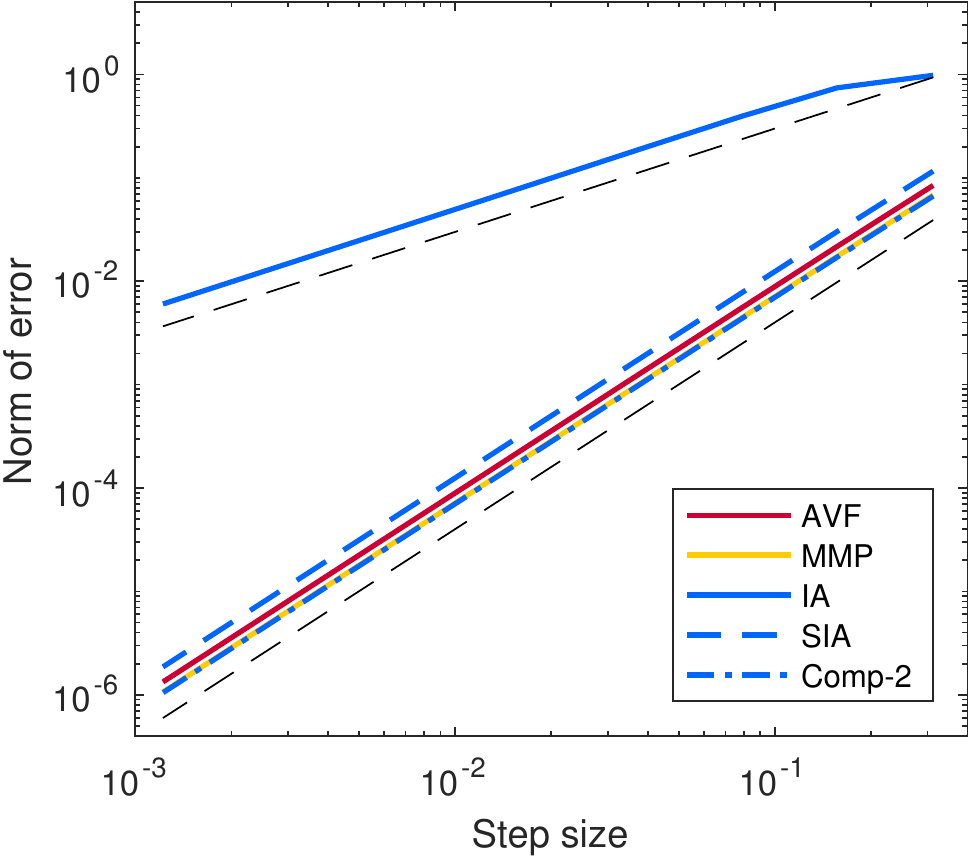}
        }\quad
        \subfloat{
                \centering
                \includegraphics[width=0.46\textwidth]{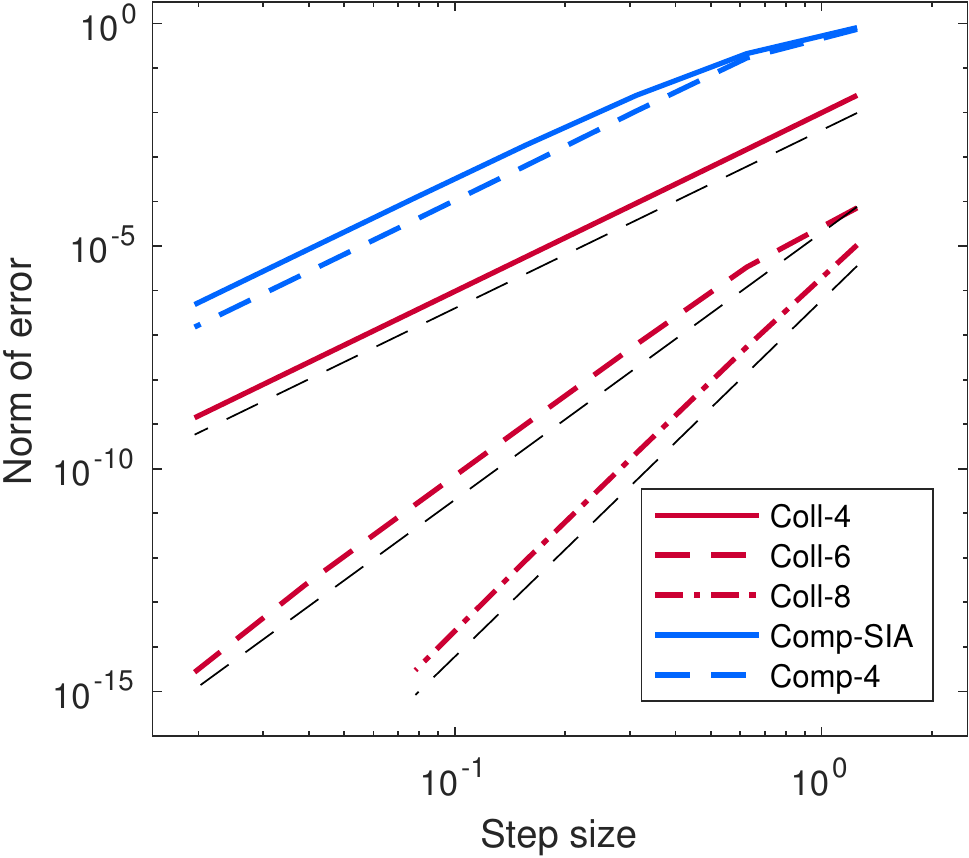}
        }
        \caption{Error norm at $t=10$ for the Heisenberg spin chain problem solved with different schemes, plotted with black, dashed reference lines of order 1, 2, 4, 6 and 8. \textit{Left:} The AVF, modified midpoint (MMP), Itoh--Abe (IA) and symmetrized Itoh--Abe (SIA) DRGs and a 3-stage composition of the IA DRG scheme (Comp-2). \textit{Right:} Collocation-type schemes of order 4, 6 and 8, a 3-stage composition of the SIA DRG scheme (Comp-SIA), and a 6-stage composition of the IA DRG scheme (Comp-4).}
        \label{fig:spin_orderplots}
\end{figure}

\section{Conclusions and further work}
We have presented a general framework for constructing energy preserving numerical integrators on Riemannian manifolds.
The main tool is to generalize the notion of discrete gradients as known from the literature. The new methods make use of an approximation to the Riemannian gradient coined the discrete Riemannian gradient, as well as a retraction map and a coordinate center function. An appealing feature of the new methods is that they do not depend on a particular choice of local coordinates or on an embedding of the manifold into a (larger) Euclidean space, but are of an intrinsic nature. Particular examples of discrete Riemannian gradient methods are given as generalizations of well-known schemes, such as the average vector field method, the midpoint discrete gradient method and the Itoh--Abe method. Extensions to higher order are proposed via a collocation-like method. We have analysed the local and global error behaviour of the methods, and they have been implemented and tested for certain spin systems where the phase space is $\left(S^2 \right)^d$. 

Possible directions for future research include a more detailed study of the stability and propagation of errors, taking into account particular features of the Riemannian manifold; for instance, it may be expected that the sectional curvature will play an important role. More examples should also be tried out, and we belive, inspired by \cite{Celledoni18}, that there is a potential for making our implementations more efficient by tailoring them for the particular manifold, as well as the ODE problem considered.

\bibliography{bibl}

\begin{thebibliography}{10}

\bibitem{Gonzalez96}
O.~Gonzalez, ``Time integration and discrete {H}amiltonian systems,'' {\em J.
  Nonlinear Sci.}, vol.~6, no.~5, pp.~449--467, 1996.

\bibitem{McLachlan99}
R.~I. McLachlan, G.~R.~W. Quispel, and N.~Robidoux, ``Geometric integration
  using discrete gradients,'' {\em R. Soc. Lond. Philos. Trans. Ser. A Math.
  Phys. Eng. Sci.}, vol.~357, no.~1754, pp.~1021--1045, 1999.

\bibitem{Crouch93}
P.~E. Crouch and R.~Grossman, ``Numerical integration of ordinary differential
  equations on manifolds,'' {\em Journal of Nonlinear Science}, vol.~3, no.~1,
  pp.~1--33, 1993.

\bibitem{Munthe95}
H.~Munthe-Kaas, ``Lie--{B}utcher theory for {R}unge--{K}utta methods,'' {\em
  BIT Numerical Mathematics}, vol.~35, no.~4, pp.~572--587, 1995.

\bibitem{Munthe98}
H.~Munthe-Kaas, ``Runge--{K}utta methods on {L}ie groups,'' {\em BIT Numerical
  Mathematics}, vol.~38, no.~1, pp.~92--111, 1998.

\bibitem{Iserles00}
A.~Iserles, H.~Z. Munthe-Kaas, S.~P. N{\o}rsett, and A.~Zanna, ``Lie-group
  methods,'' {\em Acta Numerica}, vol.~9, pp.~215--365, 2000.

\bibitem{Celledoni14}
E.~Celledoni, H.~Marthinsen, and B.~Owren, ``An introduction to {L}ie group
  integrators -- basics, new developments and applications,'' {\em Journal of
  Computational Physics}, vol.~257, pp.~1040--1061, 2014.

\bibitem{munthe-kaas13opl}
H.~Z. Munthe-Kaas and A.~Lundervold, ``On post-{L}ie algebras, {L}ie--{B}utcher
  series and moving frames,'' {\em Found. Comput. Math.}, vol.~13, no.~4,
  pp.~583--613, 2013.

\bibitem{Leimkuhler96}
B.~Leimkuhler and G.~W. Patrick, ``A symplectic integrator for {R}iemannian
  manifolds,'' {\em Journal of Nonlinear Science}, vol.~6, no.~4, pp.~367--384,
  1996.

\bibitem{Celledoni13}
E.~Celledoni and B.~Owren, ``Preserving first integrals with symmetric {L}ie
  group methods,'' {\em Discrete Contin. Dyn. Syst.}, vol.~34, no.~3,
  pp.~977--990, 2014.

\bibitem{Hairer09}
E.~Hairer, ``Energy-preserving variant of collocation methods,'' {\em JNAIAM.
  J. Numer. Anal. Ind. Appl. Math.}, vol.~5, no.~1-2, pp.~73--84, 2010.

\bibitem{Adler02}
R.~L. Adler, J.-P. Dedieu, J.~Y. Margulies, M.~Martens, and M.~Shub, ``Newton's
  method on {R}iemannian manifolds and a geometric model for the human spine,''
  {\em IMA Journal of Numerical Analysis}, vol.~22, no.~3, pp.~359--390, 2002.

\bibitem{Celledoni18}
E.~Celledoni, S.~Eidnes, B.~Owren, and T.~Ringholm, ``Dissipative schemes on
  {R}iemannian manifolds,'' {\em arXiv preprint, arXiv:1804.08104}, 2018.

\bibitem{Harten83}
A.~Harten, P.~D. Lax, and B.~van Leer, ``On upstream differencing and
  {G}odunov-type schemes for hyperbolic conservation laws,'' {\em SIAM Rev.},
  vol.~25, no.~1, pp.~35--61, 1983.

\bibitem{quispel08anc}
G.~Quispel and D.~McLaren, ``A new class of energy-preserving numerical
  integration methods,'' {\em J. of Phys. A: Math. and Theor.}, vol.~41, no.~4,
  pp.~045206, 7, 2008.

\bibitem{Itoh88}
T.~Itoh and K.~Abe, ``Hamiltonian-conserving discrete canonical equations based
  on variational difference quotients,'' {\em Journal of Computational
  Physics}, vol.~76, no.~1, pp.~85--102, 1988.

\bibitem{brugnano10hbv}
L.~Brugnano, F.~Iavernaro, and D.~Trigiante, ``Hamiltonian boundary value
  methods (energy preserving discrete line integral methods),'' {\em J. Numer.
  Anal. Ind. Appl. Math}, vol.~5, no.~1, pp.~17--37, 2010.

\bibitem{Cohen11}
D.~Cohen and E.~Hairer, ``Linear energy-preserving integrators for {P}oisson
  systems,'' {\em BIT Numerical Mathematics}, vol.~51, no.~1, pp.~91--101,
  2011.

\bibitem{kunzinger06gge}
M.~Kunzinger, H.~Schichl, R.~Steinbauer, and J.~A. Vickers, ``Global {G}ronwall
  estimates for integral curves on {R}iemannian manifolds,'' {\em Rev. Mat.
  Complut.}, vol.~19, no.~1, pp.~133--137, 2006.

\bibitem{Frank97}
J.~Frank, W.~Huang, and B.~Leimkuhler, ``Geometric integrators for classical
  spin systems,'' {\em Journal of Computational Physics}, vol.~133, no.~1,
  pp.~160--172, 1997.

\bibitem{Lewis03}
D.~Lewis and N.~Nigam, ``Geometric integration on spheres and some interesting
  applications,'' {\em Journal of Computational and Applied Mathematics},
  vol.~151, no.~1, pp.~141--170, 2003.

\bibitem{McLachlan14}
R.~I. McLachlan, K.~Modin, and O.~Verdier, ``Symplectic integrators for spin
  systems,'' {\em Physical Review E}, vol.~89, no.~6, p.~061301, 2014.

\bibitem{McLachlan17}
R.~McLachlan, K.~Modin, and O.~Verdier, ``A minimal-variable symplectic
  integrator on spheres,'' {\em Mathematics of Computation}, vol.~86, no.~307,
  pp.~2325--2344, 2017.

\end{thebibliography}
\bibliographystyle{ieeetr}

\end{document}